\documentclass[reqno,11pt]{amsart}
\usepackage{amsmath, amsfonts, amssymb, amsthm, amscd}

\DeclareMathSymbol{\leqslant}{\mathalpha}{AMSa}{"36} % nicer `smaller or equal'
\DeclareMathSymbol{\geqslant}{\mathalpha}{AMSa}{"3E} % nicer `larger or equal'
\DeclareMathSymbol{\eset}{\mathalpha}{AMSb}{"3F}     % nicer `emptyset'
\renewcommand{\leq}{\;\leqslant\;}                   % redef. of < or =
\renewcommand{\geq}{\;\geqslant\;}                   % redef. of > or =
\newcounter{mycount}
\newenvironment{romlist}{\begin{list}{\rm(\roman{mycount})}%
   {\usecounter{mycount}\labelwidth=1cm\itemsep 0pt}}{\end{list}}

\numberwithin{equation}{section}
\newtheorem{thm}[equation]{Theorem}
\newtheorem{prop}[equation]{Proposition}

\newtheorem{lem}[equation]{Lemma}

\theoremstyle{definition}

\newcommand{\cB}{\ensuremath{\mathcal B}}
\newcommand{\cC}{\ensuremath{\mathcal C}}
\newcommand{\cD}{\ensuremath{\mathcal D}}
\newcommand{\cE}{\ensuremath{\mathcal E}}

\newcommand{\cL}{\ensuremath{\mathcal L}}

\newcommand{\cS}{\ensuremath{\mathcal S}}

\newcommand{\bbE}{{\ensuremath{\mathbb E}} }

\newcommand{\bbL}{{\ensuremath{\mathbb L}} }

\newcommand{\bbR}{{\ensuremath{\mathbb R}} }

\newcommand{\bbZ}{{\ensuremath{\mathbb Z}} }

\newcommand{\gep}{\varepsilon}       % \ge already exists...
\newcommand{\gp}{\varphi} 
\newcommand{\gr}{\rho}
\newcommand{\gz}{\zeta}

\newcommand{\gl}{\lambda}
\newcommand{\gL}{\Lambda}
\newcommand{\gs}{\sigma}

\newcommand{\bra}{\langle}
\newcommand{\ket}{\rangle}
\newcommand{\grad}{\bar \nabla}

\newcommand{\Hess}{\text{Hess}\,}
\newcommand{\n}{\eta}
\newcommand{\pd}[1]{\partial_{#1}}

\newcommand{\oo}{\infty}
\newcommand{\Lg}{L_\gL}
\newcommand{\Lc}{\cL_\gL}
\newcommand{\Ls}{\tilde{L}}
\newcommand{\ns}{\tilde{\n}}
\newcommand{\sE}{\tilde{\cE}}
\newcommand{\bcdot}{\,\cdot\,}

\newcommand{\bk}[1]{\left\bra #1 \right\ket}
\newcommand{\bkz}[1]{\left\bra #1 \right\ket_\gr}

\begin{document}
\title{Helffer-Sj\"ostrand representation for conservative dynamics}
\author{T. Bodineau, B. Graham}
\address{D\'epartement de math\'ematiques et applications, Ecole Normale Sup\'erieure, UMR 8553,
75230 Paris cedex 05, France}
\begin{abstract}
\noindent
We consider a Helffer-Sj\"ostrand representation for the correlations in canonical Gibbs measures with convex interactions under conservative Ginzburg-Landau dynamics. We investigate the rate of relaxation to equilibrium.
\\

\\

\end{abstract}
\thanks{We thank  T. Funaki, G. Giacomin, B. Helffer, S. Olla, H. Spohn for very helpful discussions.
TB acknowledges the support of the French Ministry of Education through the ANR BLAN07-2184264 grant. 
BG acknowledges the support of the Fondation Sciences Math\'ematiques de Paris.
}
\maketitle
\section{Introduction}
The complicated interactions in particle systems lead to subtle
correlations which in some cases can Êbe represented in term of
simpler quantities, e.g. Êrandom walk crossings \cite{FFS}. 
A celebrated representation for the correlations of Gibbs measures 
has been obtained by Helffer and Sj\"ostrand by means of the Witten
Laplacian \cite{H livre, HS}. In this representation, the decay of correlations is related to spectral properties and it can be
studied by using spectral theory \cite{H1,H2,H3}. 
For effective interface models, this representation triggered a probabilistic
reinterpretation of the Witten Laplacian as the generator of a
random walk coupled to the evolution of the particle system \cite{NS, DGI, GOS}. 
For a large class of effective interface models, the correlations between two sites $x,y$ can be
understood as the total time spent at $y$ by a random walk (in a random environment) starting from $x$. 
The strength of the Helffer-Sj\"ostrand representation is to relate the equilibrium
correlations to the behavior of the dynamics associated to the
model.

\medskip

In this paper, we investigate the Helffer-Sj\"ostrand representation
for the correlations in canonical Gibbs measures, i.e. measures conditioned to have a fixed mean density.
This relates the equilibrium correlations in canonical measures to the conservative Ginzburg-Landau dynamics.
In contrast to the non-conservative dynamics considered in the previous
works \cite{H livre, DGI}, fixing the total density leads to a
Witten Laplacian with a different Êstructure: the gradients are now replaced by gradient differences.
In one-dimension, we will show that for a class of Hamiltonian with
convex interactions, \eqref{eq: Hypothese}, the correlations for the Êcanonical
Gibbs measure can be interpreted as the occupation time of a random walk
evolving in a random environment  coupled to the conservative
Ginzburg-Landau dynamics. 
Furthermore, the space-time correlations of the dynamics are also encoded in the diffusive mechanism of the
random walk.Ê
For some specific dynamics, like the symmetric simple exclusion
process, a similar property for the correlations was obtained by duality \cite{Liggett}.Ê
In our model, the random walk is not a consequence of a duality property but it is reminiscent of the 
stochastic process considered in \cite{DGI}.

\medskip

We will use the Helffer-Sj\"ostrand representation to study  the relaxation of the one-dimensional conservative Ginzburg-Landau
dynamics.
Equilibrium fluctuations of the Ginzburg-Landau model and the convergence of the density field to a generalized Ornstein-Uhlenbeck process have been obtained in \cite{L,Z}.
The relaxation is also expected to occur at a microscopic scale:
for initial data $\n \in \bbR^\bbZ$ sampled from the equilibrium measure $\bkz{ \bcdot }$,
the space-time correlation is conjectured to obey the following scaling form
for large  $t$ and $|i|$ (\cite{Spohn} page 177 equation (2.14))
\begin{eqnarray}
\label{eq: relaxation}
\bkz{ \n_0(0) ; \n_i (t) }
\approx
\frac{\chi  (\gr) }{\sqrt{2 \pi \hat q (\gr) \,  t}}  \exp \left( - \frac{i^2}{2 t \; \hat q (\gr)} \right) \, ,
\end{eqnarray}
where $\hat q (\gr) $ is a diffusion coefficient and $\chi  (\gr) = \bkz{\n_0 ; \n_0}$ is the susceptibility.
The intuition behind the scaling \eqref{eq: relaxation} is that an initial fluctuation of the density at the origin will 
diffuse in the course of time. 
By combining localization techniques and spectral gap bounds, 
sharp relaxation estimates of the type \eqref{eq: relaxation}  for functions supported in a neighborhood of the origin were derived 
(in any dimension) for discrete models in \cite{BZ1,BZ2, CCR,CM,JLQY} and for continuous variables in \cite{LY}.

For continuous variable models, 
the Helffer-Sj\"ostrand representation provides an alternative way to understand \eqref{eq: relaxation} as it relates the correlation between the origin  and a site $i$ at time $t$ to the probability that the random walk starting at 0 touches the site $i$ at time $t$. 
This random walk evolves in a random environment coupled to the evolution of the particle system,
therefore its precise limiting properties are difficult to study. 
By using the general theory of Aronson, De Giorgi, Nash, Moser for uniformly elliptic second order operators, 
some estimates can be obtained on the kernel of this random walk. This leads to bounds on the correlation functions 
in \eqref{eq: relaxation} (see Section \ref{sec: Relaxation to equilibrium}).
In the spirit of \cite{GOS}, more global relaxation estimates can be obtained by using the homogenization theory of 
Kipnis, Varadhan \cite{KV} (see Section \ref{sec: Equilibrium fluctuations}).

\section{Ginzburg-Landau dynamics}

Let $\gL$ denote the one dimensional torus $\gL = (\bbZ/N\bbZ)$ with nearest neighbor edges.
We are going to study the relaxation properties of conservative Ginzburg-Landau dynamics
on $\bbR^\gL$
\begin{equation}
\label{SDE}
\forall i \in \gL, \quad
d \n_i (t) = \sum_{j =i \pm 1} \left( \frac{\partial H}{\partial \n_j} (\n) - \frac{\partial H}{\partial \n_i} (\n) \right)dt + \sqrt{2}  (d B_{(i,i+1)}(t)- d B_{(i-1,i)}(t)).
\end{equation}
where $(B_{(i,i+1)} (t))_{i \in \gL}$ denote independent standard Brownian motions associated to each edge and
we will consider Hamiltonians of the form
\begin{eqnarray}
\label{eq: Hamiltonian}
H (\n) = \sum_{i \in \gL} V_1(\n_i) + V_2(\n_i + \n_{i+1}).
\end{eqnarray}
We will require convexity assumptions on the potentials, i.e. that there are constants $C\pm$ such that for 
\begin{eqnarray}
\label{eq: Hypothese}
0<C_- \leq V_k''(\bcdot) \leq C_+, \qquad k=1,2.
\end{eqnarray}
The dynamics \eqref{SDE} can be extended on $\bbZ$ (see \cite{Fritz, Z}). 
In Sections \ref{sec: Relaxation to equilibrium} and \ref{sec: Equilibrium fluctuations},
we will investigate the relaxation properties in the infinite volume limit and  for this,
we will have to restrict to Hamiltonians without interactions, i.e.  $V_2 =0$.

\medskip

The Gibbs measure on $\bbR^\gL$ associated to $H$ will be denoted by $\mu_N$
and the canonical Gibbs measure with mean density $\gr$ by
\begin{eqnarray*}
\mu_{\gr,N} ( \bcdot) = \mu_{N} \bigg( \cdot \;  \bigg| \sum_i \n_i = |\gL| \gr \bigg) \, .
\end{eqnarray*}
We will write $\bk{\bcdot}$ to denote expectation with respect to $\mu_{\gr,N}$.

Let $\cS_{\gr,N} = \{ \n \in \bbR^\gL:\sum_i \n_i = |\gL| \gr \}$.
In $\gL$, the  dynamics \eqref{SDE} conserve the total density $\sum_{i \in \gL} \n_i$.
We will see in \eqref{ibp} that $\mu_{\gr,N}$ is a (reversible) invariant measure.

\bigskip

The generator of the dynamics \eqref{SDE} is given by
\begin{eqnarray*}
%\label{eq: generator}
\Lg = \sum_{i \in \gL} -\left(\frac{\partial}{\partial \n_{i+1}}  - \frac{\partial}{\partial \n_i} \right)^2
+ \left( \frac{\partial H}{\partial \n_{i+1}}-\frac{\partial H}{\partial \n_i} \right) 
\left( \frac{\partial}{\partial \n_{i+1}} - \frac{\partial}{\partial \n_i} \right) \, .
\end{eqnarray*}
Let $\cB = \{(i,i+1)\}_{i \in \gL}$ denote the set of oriented nearest neighbor bonds of $\gL$.
It is natural to associate a differential operator with each edge.
For $b=(i,j)\in\cB$, we will write $\pd{b}$ to denote $\partial/\partial \n_j- \partial/\partial \n_i$. We will write $\grad$ to denote the vector of all such operators: $\grad= ( \pd{b} )_{b\in\cB}$.
The generator $\Lg$ can now be written:
\[
\Lg=-\grad\cdot\grad + \grad H \cdot \grad \, .
\]
Unless otherwise stated, take $\gr$ and $N$ to be fixed quantities. 
Note that there is an integration by parts formula
\begin{equation}
\label{ibp}
\bk{f \grad g} = \bk{- g \grad f} + \bk{ fg \grad H} \, ,
\end{equation}
where $\bk{\bcdot}$ denotes the expectation with respect to $\mu_{\gr,N}$.
Thus $\Lg$ is self-adjoint with respect to $\mu_{\gr,N}$ and the dynamics are reversible.
The operator $\Lg$ has a self-adjoint extension with domain included in $\bbL^2 ( \cS_{\gr,N}, \mu_{\gr,N} )$.

\section{Correlations for the canonical measure}

\subsection{Helffer-Sj\"ostrand representation}

We will derive a formula similar to the Helffer-Sj\"ostrand
representation \cite{HS, H1, H2, H3}  for the correlations under the  canonical 
Gibbs measure $\mu_{\gr,N}$.
We follow a formalism similar to the one applied previously for the non-conservative 
(Langevin) dynamics \cite{H livre}.

The correlation between two functions $f,g$ under the measure $\mu_{\gr,N}$ is defined by
\begin{eqnarray*}
\bk{ f ; g } = \bk{ (f - \bk{ f}) (g - \bk{ g })} \, .
\end{eqnarray*}
The operator $\Lg$ has a spectral gap (see \cite{LPY, Chafai} or the Appendix for an alternative proof) and thus 
for any smooth function $g$ in $\cC^\infty_0(\cS_{\gr,N}, \bbR )$ there exists a unique inverse $u$ in $\cC^\infty_0(\cS_{\gr,N}, \bbR )$ 
such that
\begin{eqnarray*}
\Lg u = g - \bk{g}\, .
\end{eqnarray*}
Using integration by parts \eqref{ibp}, one has:
\begin{eqnarray*}
\bk{ f ; g } = \bk{ (f - \bk{ f}) \Lg u }
= \cE (f,u) \, ,
\end{eqnarray*}
where the Dirichlet form $\cE$ is defined by
\begin{eqnarray}
\label{def:dirichlet_form}
\cE (f,u) = \bk{f \Lg u} = \bk{ \grad f \cdot \grad u}\,.
\end{eqnarray}
Let $\Hess H$ denote the (bond-wise) Hessian matrix 
$$
\Hess H
= [\pd{b}\pd{c} H]_{b,c\in\cB} \, .
$$
%A further application of the integration by parts formula yields 
One has
\begin{equation}
\label{eq: commutateur}
\grad g = \grad \Lg u = \Lc \grad u 
\end{equation}
where $\Lc$ denotes the `Witten-Laplacian' operator defined on $\cC^\infty_0(\cS_{\gr,N} , \bbR^\cB)$ by,
\[
\Lc\, U(\n) = \Lg \otimes {\rm Id} \; U(\n) + (\Hess H)\cdot U(\n), \qquad U:\bbR^\gL \to \bbR^\cB \, . 
\]
Here $\Lg \otimes {\rm Id} $ denotes the identity matrix with diagonal elements equal to $\Lg$.
Note that the operator $\Lc$ is also self-adjoint in $\bbL^2 ( \cS_{\gr,N}, \mu_{\gr,N} )$.

Combining \eqref{def:dirichlet_form} and \eqref{eq: commutateur}, we therefore  have
for any $f,g$ in  $\cC^\infty_0(\cS_{\gr,N}, \bbR )$ 
\begin{equation}
\label{RWrep}
\bk{f;g}= \bk{ \grad f \, \Lc^{-1}\, \grad g }.
\end{equation}
We remark that the Witten Laplacian $\Lc$ and the identity \eqref{RWrep} could have been defined for more general Hamiltonians than \eqref{eq: Hamiltonian}, and in any dimension. In the Appendix, we exploit \eqref{RWrep} to estimate the spectral gap in dimension $d \geq 1$.
However the interpretation of $\Lc$ as the generator of stochastic dynamics 
(Section \ref{subsec: The random walk representation}) requires the assumption \eqref{eq: Hypothese} on the Hamiltonian $H$ and the one-dimensional structure.

\subsection{The random walk representation}
\label{subsec: The random walk representation}

Let $X(t)$ denote a continuous time random walk on the edges in $\cB$ with jump rates determined by the Hessian of H: $X(t)$ steps from $b$ to $c$ at rate $-\pd{b}\pd{c} H$. Thus if $X(t)=b=(i,i+1)$, the non-zero jump rate to the bond $b + k = (i+k, i+k+1)$ is given by (with $k \in \{-2,-1,1,2\}$)
\begin{eqnarray*}
%\label{eq: rates}
\begin{array}{lll}
-\pd{b}\pd{b-2}H & = &V_2''(\n_i+\n_{i-1}),\\
-\pd{b}\pd{b-1}H & = &V_1''(\n_i), \\
-\pd{b}\pd{b+1}H & = &V_1''(\n_{i+1}), \\
-\pd{b}\pd{b+2}H & = &V_2''(\n_{i+1}+\n_{i+2}).
\end{array}
\end{eqnarray*}
Thus $\Lc$ can be interpreted as the generator of the joint evolution $(\n(t),X(t))$ in $\cS_{\gr,N} \times \cB$. On this space, the representation \eqref{RWrep} has a probabilistic interpretation.
\begin{prop}\label{Prop: RW}
Let $\bbE^\n_b$ denote the expectation for the joint process starting at $\n(0)=\n$ and $X(0)=b$.
\begin{eqnarray}
\label{RWrep2}
\bk{f;g} = \bk{ \grad f  \, \Lc^{-1} \grad g } =
\int_0^\infty \; \sum_{b,c \in \cB}
\bk{ \partial_b f (\n) \; \bbE^\n_b \big( 1_{X(t) = c} \,
\partial_{c} g (\n(t)) \big) } \, dt \, .
\end{eqnarray}
\end{prop}
With no extra effort, we can extend the Hamiltonian to include terms such as $V_3(\n_i+\n_{i+1}+\n_{i+2})$, $V_4(\n_i+\n_{i+1}+\n_{i+2}+\n_{i+3})$, and so on. In this case the random walk would perform jumps from $(i,i+1)$ of length $k$ with intensities $V_k^{\prime \prime}(\n_i+\n_{i-1} + \dots +\n_{i-k+1})$ and $V_k^{\prime \prime}(\n_{i+1}+\n_{i+2} + \dots +\n_{i+k})$.
For this class of models, the variables $\n$ can be reinterpreted as the gradient field of effective interface models and therefore the representation 
\eqref{RWrep2} is equivalent to the one derived for the non-conservative Ginzburg-Landau dynamics (see for example \cite[Proposition 2.2]{DGI}). 
However, looking directly at the non-conservative dynamics changes the point of view. Below, we derive new identities for the conservative Ginzburg Landau dynamics.

It does not seem possible to extend this random walk representation to dimensions $d\geq 2$. The matrix $- \Hess H$ cannot be interpreted (in an easy way) as generating a continuous time Markov chain: many of the off-diagonal elements are negative, so they cannot be interpreted as jump rates.

\begin{proof}[Proof of Proposition \ref{Prop: RW}] 
The evolution under the semi-group generated by $\Lc$ can be rewritten as follow.
For any $g$ in $\cC^\infty_0(\cS_{\gr,N}, \bbR )$
\begin{equation*}
%\label{eq: semi}
e^{-t\Lc} \grad g(\n) =\left( \bbE_b^\n \left[ \sum_c 1_{X(t)=c} \, \pd{c} g(\n_t) \right] \right)_{b\in\cB} \, .
\end{equation*}
%Thus one can integrate \eqref{eq: semi} with respect to $t$ over $[0,\oo)$. 
Applying $\Lc$ yields,
\begin{align*}
\Lc \int_0^T &e^{-t\Lc} \grad g(\n)\,dt 
= 
\Lc \left( \int_0^T \bbE_b^\n \left[  \sum_c \, 1_{X(t)=c} \; \pd{c} g(\n_t) \right] \,dt\right)_{b\in\cB}\\
&
=\grad g (\n) -  e^{- T \Lc} \grad g(\n) \, .
\end{align*}

The space of gradients is invariant under the operator $\Lc$, and therefore under the semi-group.
Theorem \ref{thm: spectral gap} implies that
\[
\bk{ \grad  g(\n) \;  e^{-t\Lc} \grad g(\n) } \leq  \exp \left( - \frac{k C_-}{N^2} t \right) \bk{ \big( \grad g(\n) \big)^2 }\, .
\]
Taking the limit $T \to \infty$ yields
\begin{align*}
\Lc \left( \int_0^\oo \bbE_b^\n \left[  \sum_c \, 1_{X(t) = c} \; \pd{c} g(\n_t) \right] \,dt\right)_{b\in\cB} =\grad g \, .
\end{align*}
Hence
\[
\Lc^{-1} \grad g= \left( \int_0^\oo \bbE_b^\n \left[  \sum_c 1_{X (t) =c} \;  \pd{c} g(\n_t)\,dt \right]  \right)_{b\in\cB} 
\]
can be substituted into \eqref{RWrep}.
\end{proof}

The joint dynamics including the field $\n(t)$ and the random walk $X(t)$ can be extended to $\bbZ$ (see e.g. \cite{Fritz}). 
In Section \ref{sec: Relaxation to equilibrium}, we will relate the relaxation to equilibrium in $\bbZ$ 
to the fluctuation of the random walk $X(t)$.

\section{Relaxation to equilibrium}
\label{sec: Relaxation to equilibrium}

In this section, we will consider the relaxation of the dynamics on $\bbZ$.
We focus only on systems with Hamiltonians $H (\n) = \sum_{x} V(\n_x)$ and with a potential $V$ satisfying 
the convexity assumption \eqref{eq: Hypothese}.
Let $\bkz{\bcdot}$ denote the product invariant measure on $\bbZ$ with mean density $\gr$.

\medskip

The random walk representation provides a clear connection between the relaxation of the particle system and the diffusive mechanism conjectured in
\eqref{eq: relaxation}.
\begin{prop}
\label{Prop: identity}
For any $i$ and $t \geq 0$, the following identity holds
\begin{eqnarray}
\label{eq: identity}
\bkz{ V' ( \n_0(0)) ; \n_i (t) }  = \bkz{ \bbE_{(0,1)}^\n \big( 1_{X(t) = (i,i+1)} \, \big)  } \, .
\end{eqnarray}
\end{prop}
In the Gaussian case $V(x)=x^2/2$, the jump rates are uniform and the identity \eqref{eq: identity} coincides with the conjecture
\eqref{eq: relaxation}
\begin{eqnarray*}
\bkz{ \n_0(0) ; \n_i (t) } =\bbE_{(0,1)} \big( 1_{X(t) = (i,i+1)} \, \big) \simeq \frac{1}{(4 \pi t)^{1/2}} \exp \left( - \frac{i^2}{4 t} \right) \, ,
\end{eqnarray*}
where $X_t$ is just a simple random walk. 
A similar relation also holds for the space-time correlations of the symmetric simple exclusion process \cite{Liggett}.
For general potentials,  \eqref{eq: identity} confirms the conjecture \eqref{eq: relaxation} as it explicitly relates the relaxation to equilibrium 
to the relaxation of the random walk.

Relation \eqref{eq: identity} is non-symmetric, however  
upper and lower bound on the relaxation of the two-point correlation function can be obtained.
\begin{prop}
\label{Prop: bounds}
For any $i$ and $t \geq 0$, one has
\begin{equation}
\label{eq: bounds}
\frac{1}{C_+} 
\leq 
\frac{\bkz{ \n_0(0) ; \n_i (t) } }{\bkz{ \bbE_{(0,1)}^\n \big( 1_{X(t) = (i,i+1)} \, \big)  }}
\leq \frac{1}{C_-},
\end{equation}
where the constants $C_{\pm}$ were introduced in \eqref{eq: Hypothese}.
\end{prop}
The previous estimates \eqref{eq: bounds} can be turned into quantitative bounds 
using Aronson %-De Giorgi-Nash-Moser 
estimates for the transition kernel 
of strictly elliptic operators. Following \cite{GOS}, there exists $c_1,c_2$ such that  for $t>1$
\begin{eqnarray}
\label{eq: Aronson upper}
\bkz{ \bbE_{(0,1)}^\n \big( 1_{X(t) = (i,i+1)} \, \big) } 
\leq \frac{c_1}{\sqrt{t}} \exp \left( - \frac{|i|}{c_1 \sqrt{t}} \right) \, ,
\end{eqnarray}
and for $|i| \leq \sqrt{t}$
\begin{eqnarray}
\label{eq: Aronson lower}
\bkz{ \bbE_{(0,1)}^\n \big( 1_{X(t) = (i,i+1)} \, \big) } 
\geq \frac{c_2}{1 \vee \sqrt{t}} \, .
\end{eqnarray}

\begin{proof}[Proof of Proposition \ref{Prop: identity}]

First we will prove \eqref{eq: identity} for the dynamics in a finite domain $\gL$ and then we will pass to the limit $\gL \to \bbZ$ for a fixed time $t$.

Given two functions $f$ and $g$, and a time $t\geq 0$, we can apply formula \eqref{RWrep2} to $f$ and $e^{-t\Lg} g = P_t (g) $ which is the semi-group at time $t$ for the Ginzburg-Landau dynamics
\begin{eqnarray*}
\bk{ f(\n) ; e^{-t\Lg} g(\n)}
&=& 
\bk{ f(\n) ; P_t(g) (\n) } \\
&=&
\int_t^\infty \; \sum_{b,c \in \cB} \bk{ \partial_b f (\n) \; \bbE^\n_b \big( 1_{X(s) = c} \, \partial_{c} g (\n(s))  \big) } \, ds \, ,
\end{eqnarray*}
where $\bk{ \bcdot}$ refers to the finite volume measure $\mu_{\gr,N}$ with the canonical constraint.
For $f(\n)=V'(\n_0)$ and $g(\n)=\n_i$, this gives
\[
\grad f = \left\{
\begin{array}{rrrl}
\partial_b f &=&  V''(\n_0), & \qquad \text{if} \ b = (-1,0),\\
\partial_b f &=& -V''(\n_0), & \qquad \text{if} \ b = (0,1),\\
\partial_b f &=&  0, & \qquad \text{otherwise},\\
\end{array}
\right.
\]
and
\[
\grad g = \left\{
\begin{array}{rrrl}
\partial_b g &=&  1, & \qquad \text{if} \ b = (i-1,i),\\
\partial_b g &=& -1, & \qquad \text{if} \ b = (i,i+1),\\
\partial_b g &=&  0, & \qquad \text{otherwise.}\\
\end{array}
\right.
\]
Thus, 
\begin{eqnarray*}
&& \bk{ V' ( \n_0(0)) ; \n_i (t) }=\\
&& \int_t^\infty \, ds \;
\bk{ V'' ( \n_0)  \bbE^\n_{(-1,0)} \big( 1_{X(s) = (i-1,i)} \, \big) }  
+ \bk{  V'' ( \n_0) \bbE^\n_{(0,1)} \big( 1_{X(s) = (i,i+1)} \, \big) }  \\
&& \qquad - 
\bk{ V'' ( \n_0) \bbE^\n_{(-1,0)} \big( 1_{X(s) = (i,i+1)} \, \big) }  
- \bk{  V'' ( \n_0) \bbE^\n_{(0,1)} \big( 1_{X(s) = (i-1,i)} \, \big) }  \,.
\end{eqnarray*}

We introduce the function
\begin{eqnarray*}
F_s (b,\n) = \bbE_{b}^\n \big( 1_{X(s) = (i,i+1)} \, \big)
\end{eqnarray*}
which is the probability that the walk starting at the edge $b$ in an initial field $\eta$ is located at $(i,i+1)$ at time $s$.

Using translation invariance,
\begin{eqnarray*}
\bk{ V' ( \n_0(0)) ; \n_i (t) }
&=& \int_t^\infty \, ds \;
\bk{ V'' ( \n_1) \Big[ F_s ((0,1),\n) - F_s ((1,2),\n)  \Big] }    \\
&& \qquad 
+ \bk{  V'' ( \n_{0}) \Big[ F_s ((0,1),\n) - F_s ((-1,0),\n)\Big] } \, .
\end{eqnarray*}
As $\bk{ \bcdot }$ is the invariant measure, $\bk{ \Lg F_s ((0,1),\n)  } = 0$, and so
\begin{eqnarray}
\bk{ V' ( \n_0(0)) ; \n_i (t) } 
&=& \int_t^\infty \bk{ [ \Lc  F_s] \;  ((0,1),\n)  } ds  =  - \int_t^\infty \partial_s \bk{ F_s ((0,1),\n)} ds  \nonumber \\
&=& \bk{ F_t ((0,1),\n)  } - \lim_{s \to \infty} \bk{ F_s ((0,1),\n)  }.
\label{eq: finite volume}
\end{eqnarray}
The random walk is uniformly distributed in the limit $s\to\oo$, and so $\bk{ F_s ((0,1),\n)  }\to |\gL|^{-1}$ where $|\gL|$ is the number of sites in $\gL$. 

For any fixed time $t>0$, one can take the limit $\gL \to \bbZ$ in \eqref{eq: finite volume}.
This concludes the proof of \eqref{eq: identity}.
\end{proof}

\begin{proof}[Proof of Proposition \ref{Prop: bounds}]
The dynamics have a monotonicity property.
\begin{lem}
\label{lem:monotonicity}
Let $f,g:\bbR^\bbZ\to\bbR$ be two locally defined functions that are non-decreasing with respect to the coordinatewise partial order. 
For $t\geq 0$, $\bkz{f(\n(0));g(\n(t))}\geq 0$.
\end{lem}

We postpone the proof of the Lemma and first use it to deduce Proposition \ref{Prop: bounds}.
The function $h(\xi)= \frac{V'(\xi)}{C_-}  - \xi$ is non-decreasing. By Lemma \ref{lem:monotonicity},
\[
0\leq \bkz{h(\n_0(0));\n_i(t)}.
\]
Hence, by the bilinearity of covariances,
\[
\bkz{\n_0(0) ; \n_i(t)} \leq \frac{1}{C_-} \bkz{V'(\n_0(0)) ; \n_i(t)}
=
\frac{1}{C_-} \bkz{ \bbE_{(0,1)}^\n \big( 1_{X(t) = (i,i+1)} \, \big)  }.
\]
The lower bound follows in the same way.
\end{proof}

\medskip

Monotonicity properties for non-conservative dynamics have been considered in \cite{Funaki}.
We provide below an alternative proof tailored for our model.

\begin{proof}[Proof of Lemma \ref{lem:monotonicity}]
By a standard approximation argument, it is sufficient to show that for locally defined, increasing events $A$ and $B$,
\begin{equation}\label{approximation argument}
\mu_\gr(\n(0)\in A \text{ and } \n(t)\in B)\geq \mu_\gr(A)\mu_\gr(B).
\end{equation}
Let $\n,\ns$ represent two solutions to the SDE \eqref{SDE} such that
\begin{romlist}
\item $\ns(0)$ has distribution $\mu_\gr$, and
\item $\n(0)$ has distribution $\mu_\gr(\,\cdot\mid A)$.
\end{romlist}
We will write $\ns(t)\leq \n(t)$ if $\ns_i(t)\leq \n_i(t)$ for all $i\in\bbZ$.
Note that \eqref{approximation argument} holds if with probability one, $\ns(t)\leq \n(t)$.

The product measure $\mu_\gr$ only differs from the conditional measure $\mu_\gr(\,\cdot\mid A)$ on the support of $A$. By Preston's FKG inequality \cite[Theorem 3]{Preston}, $\mu_\gr(\,\cdot\mid A)$ stochastically dominates $\mu_\gr$. In other words, there is a probability measure under which $\ns(0)\leq \n(0)$ almost surely. We can then let $\ns$ and $\n$ evolve with the same driving noise $(B_b)_{b\in\cB}$. 

Let $\gp(t)=\n(t)-\ns(t)$ so that $\gp(0)\geq 0$. Note that if $\gp_i(t)<0$, then
\[
\frac{d}{dt}\gp_i(t) \geq \sum_{j\sim i} V'(\n_j(t)) - V'(\ns_i(t)) \geq \sum_{\{j\sim i\,:\, \gp(j)<0\}} C_+\gp_j(t).
\] 
By \cite[Theorem 2.1]{Z}, for any $r>0$, $\n$ and $\ns$ lie in the Hilbert space
\[
L_r^2=\{\gz :\sum_{i\in\bbZ} \, \gz_i^2 \exp(-r|i|)\leq \oo\}.
\]
We can therefore define
\[
A(t)=\sum_{i\in\bbZ} 2^{-|i|} \gp_i(t)^2 1_{\{\gp_i(t)<0\}},
\]
and calculate
\[
\frac{d}{dt}A(t) \leq \sum_{\{i\sim j\,:\,\gp_i(t),\gp_j(t)<0\}} 2^{1-|i|}C_+ \gp_i(t)\gp_j(t).
\]
Checking that for any $\gp_{i-1},\gp_i,\gp_{i+1}\in\bbR$ that
\[
2^{-|i|} \gp_i(\gp_{i-1}+\gp_{i+1}) \leq 2^{-|i-1|} \gp_{i-1}^2+2^{-|i|} \gp_{i}^2+2^{-|i+1|} \gp_{i+1}^2,
\]
we obtain $dA(t)/dt \leq 6C_+A(t)$
Applying Gronwall's lemma with
\[
A(0)=0, \qquad A(t) \leq \int_0^t 6C_+ A(s) ds,
\]
we find $A(t)=0$ for all $t\geq 0$.
\end{proof}

\section{Diffusion coefficient}
\label{sec: Equilibrium fluctuations}

In Propositions \ref{Prop: identity} and \ref{Prop: bounds}, the relaxation to equilibrium was rephrased in terms of the diffusion of the random walk $X(t)$.
We compute below a Central Limit Theorem for this random walk.

We consider the Ginzburg-Landau dynamics on $\bbZ$ starting from the equilibrium measure $\bkz{\bcdot}$ at density $\gr$.
We will show that after rescaling, the random walk $X(t)$ converges to a Brownian motion. With $\gep>0$, let $X^\gep(t)=\gep X(\gep^{-2} t)$ for $t\geq 0$. Let $Y(t)$ denote a Brownian motion with $Y(0)=0$ and variance $q (\gr) >0$ (defined in \eqref{eq:q}),
\[
\bbE_0 [Y(s)Y(t)]=(s \wedge t) \, q(\gr).
\]
Let $R_1=V''(\n_1)$ (respectively, $R_{-1}=V''(\n_0 )$) denote the jump rate of the random walk $X$ from $(0,1)$ to $(1,2)$ (respectively, $(-1,0)$).
For $i\in\bbZ$, let $\tau_i$ denote the shift operator that moves vertex $i$ to $0$,
\[
\tau_i \n_j= \n_{i+j}, \qquad j\in\bbZ.
\]
\begin{thm}
\label{thm:homog}
Over any finite time interval $[0,T]$, as $\gep\to0$, $X^\gep(t)\to Y(t)$ weakly in the Skorohod space, with
\begin{equation}
\label{eq:q}
q (\gr) =2 \inf_f \Big\{\bkz{  (1-f(\n)+f(\tau_1 \n))^2 R_1}  + \cE(f,f) \Big\}.
\end{equation}
The infimum is taken over smooth, bounded local functions $f$. By abuse of notation, $\cE(f,f)$ stands for the Dirichlet form \eqref{def:dirichlet_form} extended to $\bbZ$.
\end{thm}
This formula implies that $2C_-\leq q(\gr)  \leq 2C_+$. 
For the upper bound, set $f=1$. The lower bound follows by an application of the Cauchy-Schwarz inequality (see \cite[(4.14)]{GOS}).

\medskip

Combining Proposition \ref{Prop: identity}  and Theorem \ref{thm:homog}, we get a weaker formulation of
\eqref{eq: relaxation}.
Let $\gp$ be a  smooth function  with compact support $[-1,1]$. For any $x \in \bbR$ and $t>0$
\begin{eqnarray*}
&& \lim_{\gep \to 0}
\bkz{ V^\prime \big( \n_0 (0) \big)  ; \; \gep \sum \gp \left( \gep i \right) \n_{i + \lfloor x \gep \rfloor} \left(\frac{t}{\gep^2} \right) }\\
&& \qquad \qquad \qquad =
\frac{1}{\sqrt{2 \pi  q (\gr) \,  t}} \; \int_\bbR dy \; \gp (y-x)  \exp \left( - \frac{y^2}{2 t \; q (\gr)} \right) \, ,
\end{eqnarray*}
where $\lfloor \cdot \rfloor$ stands for the integer part.

\begin{proof}[Proof of Theorem \ref{thm:homog}]

We follow the approach of Kipnis and Varadhan (see \cite[Section 4]{GOS}) and consider the process viewed from the position of the random walk $X(t)$.
Let $\ns(t)$ denote the configuration $\n(t)$ as viewed from $X(t)$:
\[
\text{if } X(t)=(i,i+1), \qquad \ns(t)=\tau_i \n(t).
\]
Note that the edge $(0,1)$ in $\ns(t)$ corresponds to the edge $X(t)$ in $\n(t)$.

We can write the displacement of the random walk $X(t)$ as the sum of a drift term and a martingale,
\[
X(t)-X(0)= \int_0^t j(\ns(s)) ds + M(t).
\]
Here $j(\ns(s))=(R_1-R_{-1})(\ns(s))$ is the drift of the random walk $X$ at time $s$. 
The process $M(t)$ is a martingale with respect to the family of $\gs$-algebras generated by $(\n(s),X(s))_{s\in[0,t]}$;
$M(t)$ is cadlag with the same jumps as $X(t)$; $\bbE(M(t)^2)$ is therefore $t\bkz{R_{-1}+R_1}=2t\bkz{R_1}$. 
We will apply Kipnis and Varadhan's \cite[Theorem 1.8]{KV} to the drift term. To do this, there are two conditions we must check, see Lemmas \ref{KV1} and \ref{KV2}. First some notation.
Let $\Ls$ denote the generator of $\ns(t)$,
\begin{eqnarray*}
\Ls F(\n) = L F (\n) + \sum_{k=\pm1} [F(\n) - F(\tau_k \n)] R_k \, ,
\end{eqnarray*}
where $L$ stands for the generator of the Ginzburg-Landau dynamics on $\bbZ$.

Define a Dirichlet form $\sE$ by
\begin{eqnarray*}
\sE(f,f) =\bkz{ f \Ls f}.
\end{eqnarray*}
Note that by translation invariance, $\bkz{\bcdot}=\mu_\gr$ is the invariant measure corresponding to both $L$ and $\Ls$, and
\begin{eqnarray}\label{def:tE}
\sE(f,g)= \cE(f,g)+\bkz{ (f(\n)-f(\tau_1 \n))R_1(g(\n)-g(\tau_1 \n))}.
\end{eqnarray}
Define dual norms $\|\bcdot\|_1$ and $\|\bcdot\|_{-1}$ by
\begin{align*}
\|f\|_1^2=\sE(f,f), \qquad \|g\|_{-1}^2 &= \sup_f\left\{2 \bkz{fg} -\|f\|_1^2\right\}.
\end{align*}
Note that $\|g\|_{-1}=\oo$ unless $\bk{g}=0$. 
\begin{lem}\label{KV1}
The shifted process $(\ns(s))$ is time ergodic with respect to $\mu_\gr$.
\end{lem}
\begin{proof}
If $e^{-\Ls t}f=f$ for all $t\geq 0$, then $\Ls f=0$ and therefore $\sE(f,f)=0$.
As $V''\geq C_-$, 
\begin{equation}
\label{inq:tE}
0=\sE(f,f) \geq C_- \bkz{ (f(\n)-f(\tau_1 \n))^2 }.
\end{equation}
Hence $f$ is translation invariant. The stationary measure is ergodic, so $f$ is constant with probability one.
\end{proof}
\begin{lem}
\label{KV2} 
The drift function is finite in the dual norm: $\|j\|_{-1}<\oo$.
\end{lem}
\begin{proof}
A sufficient condition \cite[1.14]{KV} for $\|j\|_{-1}<\oo$ is that, 
\[
\forall f\in\cD(\Ls), \qquad \bkz{ f j } \leq C \|f\|_1.
\]
By \eqref{inq:tE} and translation invariance,
\begin{align*}
\bkz{  f(\n) j(\n) }  &= \sum_{k=\pm1} \bkz{  f(\n) k R_k} = \bkz{ R_1 (f(\n)-f(\tau_1 \n))}\\
&\leq  \bkz{ R_1^2}^{1/2} \bkz{  (f(\n)-f(\tau_1 \n))^2} ^{1/2}\\
&\leq C_+ \bkz{(f(\n)-f(\tau_1 \n))^2} ^{1/2}\\
&\leq C_+ C_-^{-1/2} \|f\|_1.\qedhere
\end{align*}
\end{proof}
We now seek to determine the diffusion coefficient $q(\gr)$. The random walk $X(t)$ is antisymmetric \cite{MFGW}; on the time interval $[0,T]$, the law of $(X(t),j(t))$ is equal to the law of $(X(T-t)-X(T),j(T-t))$, but the displacement of the random walk part is in the opposite direction. This symmetry implies that,
\[
\bbE\left[ X(t) \int_0^t j(s) ds\right] = 0.
\]
This allows us to expand $\frac1t\bbE(M_t^2)= \frac1t \bbE((X(t)-\int_0^t j(\ns(s)) \, ds)^2)$:
\[
2\bkz{R_1}= \frac1t \bbE( X(t)^2)  + \frac1t\bbE\left[ \left(\int_0^t j(\ns(s))ds\right)^2\right].
\] 
By \cite[Remark 1.7 and Theorem 1.8]{KV},
\[
\lim_{t \to \infty} \frac{1}{t} \bbE \left[ \left( \int_0^t j(\ns(s)) ds \right)^2 \right]=2\|j\|_{-1}^2.
\]
By $\eqref{def:tE}$, 
\begin{align*}
%\label{eq: q egalites}
&q(\gr) = 2 \bkz{R_1}  -2\sup_f\left\{2\bkz{ f j(\n) }  - \sE(f,f)  \right\}\\
&=  2\inf_f\left\{\bkz{R_1}-2\bkz{ f (R_{1}-R_{-1}) }  +\bkz{ (f(\n)-f(\tau_1 \n))^2 R_1} +\cE(f,f) 
\right\}%\nonumber
\end{align*}
The variational formula \eqref{eq:q} now follows as $\bkz{fR_{-1}}=\bkz{f(\tau_1\n)R_1}$.
\end{proof}

\section{Conclusion}

In this paper, we derived the Helffer-Sj\"ostrand representation for the equilibrium correlations in canonical Gibbs measures. 
For a class of one-dimensional Hamiltonian, this representation can be reinterpreted in terms of a stochastic process
describing the joint evolution of the conservative Ginzburg-Landau dynamics and a random walk coupled to this dynamics.
Using the random walk analogy, the diffusive relaxation of the Ginzburg-Landau dynamics 
can be related to the diffusive behavior of this random walk (Proposition \ref{Prop: identity}). 
In this way several bounds on the return to equilibrium for the Ginzburg-Landau dynamics are obtained 
\eqref{eq: Aronson upper}, \eqref{eq: Aronson lower}.

To sharpen the estimates in this paper and to derive the precise relaxation to equilibrium conjectured in \eqref{eq: relaxation},
one would need to prove a local central limit theorem for the random walk in the Helffer-Sj\"ostrand representation.
This  seems to be a challenging task.
It would be also interesting to obtain a stochastic interpretation of the Witten Laplacian in higher dimensions.
This would provide a straightforward approach to derive relaxation bounds for Ginzburg-Landau dynamics
 in dimension $d \geq 2$.

\appendix
\section{Spectral Gap}

Sharp bounds on the spectral gap have been derived for conservative
Ginzburg-Landau type dynamics in \cite{LPY, Chafai, Caputo}.
In this appendix, we show how to recover these bounds by using the
Witten Laplacian formalism when the potential is strictly convex. We
will follow the probabilistic approach devised in \cite{Ledoux}.

\medskip

In this appendix, we will relax some assumptions on the dynamics and we suppose that it is defined in dimension $d \geq 1$.
Let $\gL$ denote the $d$-dimensional torus $(\bbZ/N\bbZ)^d$.
Let $\cB$ denote the set of oriented nearest neighbor edges of $\gL$; how the edges are oriented will not be important. 
With $\grad=(\pd{b})_{b\in\cB}$, the definitions of $\Lg$ and $\Lc$ extend naturally to this higher dimensional setting.
Consider the Hamiltonian
\[
H(\n)=H_1(\n)+H_2(\n)= \sum_x V_1(\n_x) + \sum_{(x,y)\in\cB} V_2(\n_x+\n_y) \,.
\]
We will assume that $V_1$ is strictly convex, with $V_1''\geq C_->0$,
and that $V_2$ is convex.

Let $\gl=\gl(\gr,N,d)$ denote the spectral gap of the operator $\Lg$
with respect to $\mu_{\gr,N}$,
\[
\gl = \inf_{f  \perp 1} \frac{\cE(f,f)}{\bk{f ; f}}.
\]
The Dirichlet form $\cE$ was introduced in \eqref{def:dirichlet_form}.
Define also 
\[
\tilde \gl = \inf_{f  \perp 1} \frac{\bk{ \grad f \, \cdot \,  \Lc \grad f}}{\bk{ \grad f \, \cdot \,  \grad f}}.
\]

\begin{thm}
\label{thm: spectral gap}
There is a constant $k$ such that for all $N$ and $\gr$, 
\[
\gl \geq \tilde \gl \geq  \frac{kC_-}{N^2} \, .
\]
\end{thm}
The scaling $N^2$ is of the correct order as it characterizes the diffusive behavior of the conservative dynamics.
The assumption of strict convexity should only be seen as a limitation in the method of proof.
\begin{proof}
We will first show that $\gl\geq\tilde\gl$. 
Recall that $P_t=e^{-t\Lg}$ denotes the semi-group associated with the dynamics. 
As $P_0 f=f$ and $P_\oo f = \bk{f}$,
\begin{align*}
\bk{f;f}&=\bk{f[P_0 f - P_\oo f]}=\int_{0}^\oo \bk{f  \Lg P_t f}dt\\
&=\int_{0}^\oo \bk{P_{t/2}f  \Lg P_{t/2} f}dt
= \int_{0}^\oo \cE(P_{t/2}f,P_{t/2}f)dt \, .
\end{align*}
Let $F(t)=\cE(P_t f,P_tf)$. Then by \eqref{eq: commutateur}
\[
F'(t)=  -2 \bk{\grad P_t f \cdot \grad \Lg P_t f }=-2\bk{ \grad P_t f \cdot \Lc \grad P_t f }.
\]
For any $u$,
\[
\bk{ \grad u \cdot \Lc \grad u }
\geq \tilde\gl \bk{\grad u \cdot \grad u}
=\tilde\gl \cE(u,u).
\]
With $u=P_t f$, $F(t)\leq \exp(-2t\tilde\gl)F(0)$ and the bound on $\gl$ follows:
\[
\bk{f;f} = \int_0^\oo F(t/2) dt \leq \int_0^\oo \exp(-t\tilde\gl) F(0)
dt = \tilde\gl^{-1}\cE(f,f).
\]
\medskip

\noindent
We will now show the lower bound for $\tilde\gl$. With $F=\grad f$ we need,
\[
\bk{ F \cdot \Lc F }
=\bk{ F \cdot \Lg \otimes {\rm Id}\, F } + \bk{ F \cdot (\Hess H) F} \ge \frac{kC_-}{N^2}\bk{F\cdot F}.
\]
The term $\bk{ F \cdot \Lg \otimes {\rm Id}\, F }$ is equal to $\sum_b
\cE(\partial_b f,\partial_b f)\geq 0$, thus
\begin{eqnarray*}
\bk{ F \cdot \Lc F }
\geq
\bk{ F \cdot (\Hess H) F}
\end{eqnarray*}
Let $S$ denote either $\{x\}$ (for $x\in\gL$) or $\{x,y\}$ (for $(x,y)\in\cB$).
If $S=\{x\}$, let $W=V_1(\n_x)$; if $S=\{x,y\}$, let $W= V_2(\n_x+\n_y)$. 
Let $v_S(b)=+1$ if $b$ points in to $S$, let $v_S(b)=-1$ if $b$ points out from $S$, and let $v_S(b)=0$, if $b$ points neither into or out from $S$. Note that $\pd{b}\pd{c}W=v_S(b)v_S(c)W''$ and
\[
F\cdot (\Hess W) F \geq C_- \biggl(\sum_{b} v_S(b) \pd{b} f\biggr)^2,
\]
Hence, 
\begin{eqnarray*}
F \cdot (\Hess H_1) F &\geq&
C_- \sum_{x}  \left( \sum_{i=1}^d\pd{(x,x+e_i)}f -\pd{(x-e_i,x)}f \right)^2
\end{eqnarray*}
and $F \cdot (\Hess H_2) F\geq 0$.
It is sufficient now to show that,
\begin{equation}
\label{eq:spectral_gap}
\sum_x  \left( \sum_{i=1}^d\pd{(x,x+e_i)}f -\pd{(x-e_i,x)}f \right)^2
\geq \frac{k}{N^2} (F\cdot F)\,.
\end{equation}
Let $\gp_x=\partial f /\partial \n_x$, 
and let 
\[
g(\gp)= \sum_x \left( \sum_{i=1}^d -\gp_{x-e_i}+2\gp_x -
\gp_{x+e_i}\right)^2, \qquad h(\gp) = \sum_{(x,y)\in\cB} (\gp_x
-\gp_y)^2.
\]
Inequality \eqref{eq:spectral_gap} is equivalent to $g(\gp)\geq (k/N^2) h(\gp)$.
Let $Q=[q_{xy}]$ denote the generator matrix for the rate-$1$ nearest-neighbor simple random walk on $\gL$: $q_{xy}=1$ iff $x\sim y$ and $\sum_y q_{xy}=0$.
$Q$ has $|\gL|$ eigenvalues $0=\lambda_1>\lambda_2\geq \dots \geq\lambda_{|\gL|}$, and corresponding eigenvectors $\psi_1,\dots,\psi_{|\gL|}$. 
$Q$ has a uniform stationary distribution, so we can assume that the eigenvectors form an orthonormal basis for $\ell_2(\bbR^{\gL})$. 
The spectral gap of the walk $|\lambda_2|$ is known to be at least $k/N^2$ with $k$ a constant.
Write $\gp=\sum_{i=1}^{|\gL|} a_i \psi_i$. 
Then
\[
g\left(\gp\right) = \sum_x (\gp Q) (x)^2 = \sum_{i=2}^{|\gL|} a_i^2\lambda_i^2,
\]
and
\[
h(\gp) = -\sum_x \gp(x)  (\gp Q) (x) \geq |\lambda_2| \sum_{i=2}^{|\gL|} a_i^2\,, 
\]
and so $g(\gp)\geq |\lambda_2| h(\gp)$.
\end{proof}

\end{document}